\documentclass{article}
\usepackage{amssymb}
\usepackage{amsmath}

\newtheorem{theorem}{Theorem}

\newtheorem{corollary}{Corollary}

\newtheorem{lemma}{Lemma}

\newtheorem{remark}{Remark}

\newenvironment{proof}[1][Proof]{\noindent\textbf{#1.} }{\ \rule{0.5em}{0.5em}}

\begin{document}

\title{\textbf{Approximation of functions from }$L^{p}\left( \omega \right)
_{\beta }$\textbf{\ by matrix means of their Fourier series}}
\author{\textbf{Rados\l awa Kranz,} \textbf{W\l odzimierz \L enski \ and
Bogdan Szal} \\
University of Zielona G\'{o}ra\\
Faculty of Mathematics, Computer Science and Econometrics\\
65-516 Zielona G\'{o}ra, ul. Szafrana 4a, Poland\\
R.Kranz@wmie.uz.zgora.pl ,\\
W.Lenski@wmie.uz.zgora.pl ,\\
B.Szal @wmie.uz.zgora.pl}
\date{}
\maketitle

\begin{abstract}
We formulate some special conditions for the integrable functions and moduli
of continuity. We give the results on rate of approximation of such
functions by matrix means of their Fourier series, where the entries of the
rows of the matrix generate the sequences belonging to the classes $MRBVS$
and $MHBVS$. We also present some results on norm approximation for
functions from the generalized integral Lipschitz classes.

\ \ \ \ \ \ \ \ \ \ \ \ \ \ \ \ \ \ \ \ 


\ \ \ \ \ \ \ \ \ \ \ \ \ \ \ \ \ \ \ 

\end{abstract}

\section{Introduction}

Let $L^{p}\ (1\leq p<\infty )\;$ be the class of all $2\pi $--periodic
real--valued functions integrable in the Lebesgue sense with $p$--th power
over $Q=$ $[-\pi ,\pi ]$ with the seminorm 
\begin{equation}
\Vert f\Vert :=\Vert f(\cdot )\Vert _{_{L^{p}}}=\left( \int_{_{_{Q}}}\mid
f(t)\mid ^{p}dt\right) ^{1/p}  \label{1}
\end{equation}%
and consider the trigonometric Fourier series 
\begin{equation*}
Sf(x):=\frac{a_{0}(f)}{2}+\sum_{\nu =1}^{\infty }(a_{\nu }(f)\cos \nu
x+b_{\nu }(f)\sin \nu x)
\end{equation*}%
with the partial sums \ $S_{k}f$.

Let $A:=\left( a_{n,k}\right) $ be an infinite lower triangular matrix of
real numbers such that 
\begin{equation*}
a_{n,k}\geq 0\text{ when \ }k=0,1,2,...n\text{, \ }a_{n,k}=0\text{\ when }k>n%
\text{ ,}
\end{equation*}%
\begin{equation*}
\sum_{k=0}^{n}a_{n,k}=1\text{, \ where }n=0,1,2,...\text{\ ,}
\end{equation*}%
and let the \ $A-$transformation$\ $of \ $\left( S_{k}f\right) $ be given by 
\begin{equation*}
T_{n,A}^{\text{ }}f\left( x\right) :=\sum_{k=0}^{n}a_{n,k}S_{k}f\left(
x\right) \text{ \ \ \ }\left( n=0,1,2,...\right) .
\end{equation*}%
Denote, for $m=0,1,2,..,n,$%
\begin{equation*}
A_{n,m}=\sum_{k=0}^{m}a_{n,k}\text{ \ \ and \ }\overline{A}%
_{n,m}=\sum\limits_{k=m}^{n}a_{n,k}.
\end{equation*}%
We define two classes of sequences (see \cite{L}).

A sequence $c:=\left( c_{k}\right) $ of nonnegative numbers tending to zero
is called the Rest Bounded Variation Sequence, or briefly $c\in RBVS$, if it
has the property 
\begin{equation*}
\sum\limits_{k=m}^{\infty }\left\vert c_{k}-c_{k+1}\right\vert \leq K\left(
c\right) c_{m}
\end{equation*}
for all positive integer $m$, where $K\left( c\right) $ is a constant
depending only on $c$.

A sequence $c:=\left( c_{k}\right) $ of nonnegative numbers will be called
the Head Bounded Variation Sequence, or briefly $c\in HBVS$, if it has the
property 
\begin{equation*}
\sum\limits_{k=0}^{m-1}\left\vert c_{k}-c_{k+1}\right\vert \leq K\left(
c\right) c_{m}
\end{equation*}
for all positive integer $m$, or only for all $m\leq n$ if the sequence $c$
has only finite nonzero terms and the last nonzero term is $c_{n}$.

Now, we define two another classes of sequences.

Follows by L. Leindler (see \cite{L2}) a sequence $c:=\left( c_{k}\right) $
of nonnegative numbers tending to zero is called the Mean Rest Bounded
Variation Sequence, or briefly $c\in MRBVS$, if it has the property 
\begin{equation}
\sum\limits_{k=m}^{\infty }\left\vert c_{k}-c_{k+1}\right\vert \leq K\left(
c\right) \frac{1}{m+1}\sum\limits_{k\geq m/2}^{m}c_{k}  \label{2}
\end{equation}%
for all positive integer $m$.

Similar classes of sequences were considered in the papers \cite{DT}, \cite%
{LT} and \cite{T}.

Analogously, a sequence $c:=\left( c_{k}\right) $ of nonnegative numbers
will be called the Mean Head Bounded Variation Sequence, or briefly $c\in
MHBVS$, if it has the property 
\begin{equation}
\sum\limits_{k=0}^{n-m-1}\left| c_{k}-c_{k+1}\right| \leq K\left( c\right) 
\frac{1}{m+1}\sum\limits_{k=n-m}^{n}c_{k},  \label{3}
\end{equation}
for all positive integer $m<n$,\ where the sequence $c$ has only finite
nonzero terms and the last nonzero term is $c_{n}$.

It is clear that (see \cite{S}) 
\begin{equation*}
RBVS\varsubsetneq MRBVS\text{ \ \ and \ \ }HBVS\varsubsetneq MHBVS.
\end{equation*}

Consequently, we assume that the sequence $\left( K\left( \alpha _{n}\right)
\right) _{n=0}^{\infty }$ is bounded, that is, that there exists a constant $%
K$ such that 
\begin{equation*}
0\leq K\left( \alpha _{n}\right) \leq K
\end{equation*}
holds for all $n$, where $K\left( \alpha _{n}\right) $ denote the sequence
of constants appearing in the inequalities (\ref{2}) or (\ref{3}) for the
sequence $\alpha _{n}=\left( a_{n,k}\right) _{k=0}^{n}$, $n=0,1,2...$.

Now we can give the conditions to be used later on. We assume that for all $%
n $ and $0\leq m<n$%
\begin{equation}
\sum\limits_{k=m}^{n-1}\left\vert a_{n,k}-a_{n,k+1}\right\vert \leq K\frac{1%
}{m+1}\sum\limits_{k\geq m/2}^{m}a_{n,k}  \label{4}
\end{equation}
and 
\begin{equation}
\sum\limits_{k=0}^{n-m-1}\left\vert a_{n,k}-a_{n,k+1}\right\vert \leq K\frac{%
1}{m+1}\sum\limits_{k=n-m}^{n}a_{n,k}  \label{5}
\end{equation}
hold if $\left( a_{n,k}\right) _{k=0}^{n}$ belongs to $MRBVS$ or $MHBVS$,
for $n=1,2,...$, respectively.

As a measure of approximation of $f$ by $T_{n,A}^{\text{ }}f$ we use the
generalized modulus of continuity of $f$ in the space $L^{p}$ defined for $%
\beta \geq 0$ by the formula 
\begin{equation*}
\omega _{\beta }f\left( \delta \right) _{L^{p}}:=\sup_{0\leq \left\vert
t\right\vert \leq \delta }\left\{ \left\vert \sin \frac{t}{2}\right\vert
^{\beta p}\int\limits_{0}^{\pi }\left\vert \varphi _{x}\left( t\right)
\right\vert ^{p}dx\right\} ^{\frac{1}{p}},
\end{equation*}%
where 
\begin{equation*}
\varphi _{x}\left( t\right) :=f\left( x+t\right) +f\left( x-t\right)
-2f\left( x\right) .
\end{equation*}%
It is clear that for $\beta \geq \alpha \geq 0$%
\begin{equation*}
\text{ }\omega _{\beta }f\left( \delta \right) _{L^{p}}\leq \omega _{\alpha
}f\left( \delta \right) _{L^{p}}
\end{equation*}%
and it is easily seen that $\omega _{0}f\left( \cdot \right) _{L^{p}}=\omega
f\left( \cdot \right) _{L^{p}}$ is the classical modulus of smoothness.

In our theorems we shall consider the pointwise deviation $T_{n,A}^{\text{ }%
}f\left( x\right) -f\left( x\right) $. We shall formulate the special
conditions for the functions and moduli of continuity. Consequently, we also
give some results on norm approximation. We shall use the matrix which the
entries of the rows generate the sequences belonging to the classes of
sequences $MRBVS$ and $MHBVS.$ Thus we essentially extend and improve the
earlier results \cite[Theorem 2, p. 347]{SL} and \cite{LS}. Moreover, we\
prove that from our results the pointwise version of the correct form of the
Lal theorem follows.

We shall write $I_{1}\ll I_{2}$ if there exists a positive constant $K$,
sometimes depending on some parameters, such that $I_{1}\leq KI_{2}$.

\section{Statement of the results}

Let us consider a function $\omega $ of modulus of continuity type on the
interval $[0,2\pi ],$ i.e. a nondecreasing continuous function having the
following properties:\ $\omega \left( 0\right) =0,$\ $\omega \left( \delta
_{1}+\delta _{2}\right) \leq \omega \left( \delta _{1}\right) +\omega \left(
\delta _{2}\right) $\ for any\ $0\leq \delta _{1}\leq \delta _{2}\leq \delta
_{1}+\delta _{2}\leq 2\pi $. It is easy to conclude that the function \ $%
\delta ^{-1}\omega \left( \delta \right) $\ is quasi nonincreasing function
of \ $\delta .$\ Let 
\begin{equation*}
L^{p}\left( \omega \right) _{\beta }=\left\{ f\in L^{p}:\omega _{\beta
}f\left( \delta \right) _{L^{p}}\ll \omega \left( \delta \right) \right\} .
\end{equation*}
It is clear that, for $\beta >\alpha \geq 0,$%
\begin{equation*}
\text{ }L^{p}\left( \omega \right) _{\alpha }\subset L^{p}\left( \omega
\right) _{\beta }\text{ .}
\end{equation*}
Now, we can formulate our main results on the degrees of pointwise
summability. The pointwise nature of our results ensure the assumptions of
presented below theorems with $f\left( x\right) $ (cf. \cite{BL}).

\begin{theorem}
Let \ $f\in L^{p}$ $\left( 1<p<\infty \right) $ and let $\omega $ satisfy 
\begin{equation}
\sum\limits_{m=1}^{n}(m+1)^{\beta +1-\frac{2}{q}}\left\{ \int_{\frac{\pi }{%
m+1}}^{\frac{\pi }{m}}\left( \frac{\left\vert \varphi _{x}\left( t\right)
\right\vert }{\omega \left( t\right) }\right) ^{p}\sin ^{\beta p}\frac{t}{2}%
dt\right\} ^{1/p}=O_{x}\left( \left( n+1\right) ^{\beta +1/p}\right)
\label{2.6}
\end{equation}%
and 
\begin{equation}
\left\{ \int_{0}^{\frac{2\pi }{n}}\left( \frac{\left\vert \varphi _{x}\left(
t\right) \right\vert }{\omega \left( t\right) }\right) ^{p}\sin ^{\beta p}%
\frac{t}{2}dt\right\} ^{1/p}=O_{x}\left( \left( n+1\right) ^{-1/p}\right)
\label{2.7}
\end{equation}%
with $0\leq \beta <1-\frac{1}{p}$ and $q=\frac{p}{p-1}.$ If $\left(
a_{n,k}\right) _{k=0}^{n}\in MRBVS,$ then 
\begin{equation*}
\left\vert T_{n,A}^{\text{ }}f\left( x\right) -f\left( x\right) \right\vert
=O_{x}\left( \left( n+1\right) ^{\beta +\frac{1}{p}}\sum_{k=0}^{n}a_{n,k}%
\omega \left( \frac{\pi }{k+1}\right) \right)
\end{equation*}%
and if $\left( a_{n,k}\right) _{k=0}^{n}\in MHBVS,$ then 
\begin{equation*}
\left\vert T_{n,A}^{\text{ }}f\left( x\right) -f\left( x\right) \right\vert
=O_{x}\left( \left( n+1\right) ^{\beta +\frac{1}{p}}\sum_{k=0}^{n}a_{n,n-k}%
\omega \left( \frac{\pi }{k+1}\right) \right) ,
\end{equation*}%
for considered\ $x.$
\end{theorem}

Supposing a stronger condition than (\ref{2.6}) we can obtain a better order
of approximation.

\begin{theorem}
Let \ $f\in L^{p}$ $\left( 1<p<\infty \right) $ and let $\omega $ satisfy 
\begin{equation}
\sum\limits_{m=1}^{n}(m+1)^{\beta +1-\frac{2}{q}}\left\{ \int_{\frac{\pi }{%
m+1}}^{\frac{\pi }{m}}\left( \frac{\left\vert \varphi _{x}\left( t\right)
\right\vert }{\omega \left( t\right) }\right) ^{p}\sin ^{\beta p}\frac{t}{2}%
dt\right\} ^{1/p}=O_{x}\left( \left( n+1\right) ^{\beta }\right)  \label{2.8}
\end{equation}%
and$\ $ (\ref{2.7}) with $0\leq \beta <1-\frac{1}{p}$ and $q=\frac{p}{p-1}.$
If $\left( a_{n,k}\right) _{k=0}^{n}\in MRBVS,$ then 
\begin{equation*}
\left\vert T_{n,A}^{\text{ }}f\left( x\right) -f\left( x\right) \right\vert
=O_{x}\left( \left( n+1\right) ^{\beta }\sum_{k=0}^{n}a_{n,k}\omega \left( 
\frac{\pi }{k+1}\right) \right)
\end{equation*}%
and if $\left( a_{n,k}\right) _{k=0}^{n}\in MHBVS,$ then 
\begin{equation*}
\left\vert T_{n,A}^{\text{ }}f\left( x\right) -f\left( x\right) \right\vert
=O_{x}\left( \left( n+1\right) ^{\beta }\sum_{k=0}^{n}a_{n,n-k}\omega \left( 
\frac{\pi }{k+1}\right) \right) ,
\end{equation*}%
for considered\ $x.$
\end{theorem}

Next, we formulate the result on estimates in the $L^{p}$ norm of the
deviation considered above.

\begin{theorem}
Let \ $f\in L^{p}\left( \omega \right) _{\beta }$ $\left( 1<p<\infty \right) 
$\ with $0\leq \beta <1-\frac{1}{p}$. If $\left( a_{n,k}\right)
_{k=0}^{n}\in MRBVS$,\ then 
\begin{equation*}
\left\Vert T_{n,A}^{\text{ }}f\left( \cdot \right) -f\left( \cdot \right)
\right\Vert _{L^{p}}=\left\{ 
\begin{array}{c}
O\left( \left( n+1\right) ^{\beta }\sum_{k=0}^{n}a_{n,k}\omega \left( \frac{%
\pi }{k+1}\right) \right) \text{ \ \ \ \ for \ \ }\beta >0, \\ 
O\left( \left( n+1\right) ^{\frac{1}{p}}\sum_{k=0}^{n}a_{n,k}\omega \left( 
\frac{\pi }{k+1}\right) \right) \text{ \ \ \ \ for \ \ }\beta =0,%
\end{array}%
\right.
\end{equation*}%
and if $\left( a_{n,k}\right) _{k=0}^{n}\in MHBVS$,\ then 
\begin{equation*}
\left\Vert T_{n,A}^{\text{ }}f\left( \cdot \right) -f\left( \cdot \right)
\right\Vert _{L^{p}}=\left\{ 
\begin{array}{c}
O\left( \left( n+1\right) ^{\beta }\sum_{k=0}^{n}a_{n,n-k}\omega \left( 
\frac{\pi }{k+1}\right) \right) \text{ \ \ \ \ for\ \ \ }\beta >0, \\ 
O\left( \left( n+1\right) ^{\frac{1}{p}}\sum_{k=0}^{n}a_{n,n-k}\omega \left( 
\frac{\pi }{k+1}\right) \right) \text{ \ \ \ \ for\ \ \ }\beta =0.%
\end{array}%
\right.
\end{equation*}
\end{theorem}

Finally, we give some remarks and corollary.

\begin{remark}
Analyzing the proofs of our Theorems we can deduce that under the additional
assumption $A_{n,k}=O\left( \frac{k}{n+1}\right) $ or $\overline{A}%
_{n,k}=O\left( \frac{k}{n+1}\right) $\ we obtain the expression $\omega
\left( \frac{\pi }{n+1}\right) $ instead of $\ \sum_{k=0}^{n}a_{n,k}\omega
\left( \frac{\pi }{k+1}\right) $ or $\sum_{k=0}^{n}a_{n,n-k}\omega \left( 
\frac{\pi }{k+1}\right) ,$ respectively, in the estimates of our deviation $%
T_{n,A}^{\text{ }}f-f$ .
\end{remark}

\begin{remark}
We can observe that taking $a_{n,k}=\frac{1}{n+1}\sum_{\nu =k}^{n}\frac{%
p_{\nu -k}}{P_{\nu }}$ we obtain the mean considered in \cite{SL} and if $%
\left( p_{\nu }\right) $ is monotonic with respect to $\nu $ then $\left(
a_{n,k}\right) _{k=0}^{n}\in MRBVS.$ Moreover, under the assumptions of
Theorem 1 on a function $f$, if $\left( p_{\nu }\right) $ is \textit{%
nonincreasing sequence such that the condition}%
\begin{equation*}
P_{\tau }\sum\limits_{\nu =\tau }^{n}P_{\nu }^{-1}=O\left( \tau \right) 
\end{equation*}%
\textit{holds,} then from Remark 1, using Lemma 3, we can obtain the correct
form of the result of S. Lal \cite{SL}.
\end{remark}

\begin{corollary}
If $f\in Lip\left( \alpha ,p\right) :=\left\{ g\in L^{p}:\omega g\left(
\delta \right) _{L^{p}}\ll \delta ^{\alpha }\right\} $ $\left( 0<\alpha \leq
1,\text{ }1\leq p<\infty \right) $, then $f\in L^{p}\left( \omega \right)
_{\beta }$ with $\omega \left( \delta \right) =\delta ^{\alpha +\beta },$
where $\alpha +\beta \leq 1$. Taking the Fejer mean, i.e. $a_{n,k}=\frac{1}{%
n+1}$ for $k=0,1,...,n$ and $a_{n,k}=0$ for $k>n$, we get from Theorem 3, by
Remark 1 with $\alpha \in (0,1),$ the following estimate%
\begin{equation*}
\left\Vert T_{n,\left( \frac{1}{n+1}\right) }^{\text{ }}f\left( \cdot
\right) -f\left( \cdot \right) \right\Vert _{L^{p}}=O\left( \frac{1}{\left(
n+1\right) ^{\alpha }}\right) 
\end{equation*}%
for $\beta \in \left( 0,1-\frac{1}{p}\right) $, where $1<p<\infty $. So, we
obtain the best order of approximation on wider class of functions.
\end{corollary}

We can also deduce further remarks.

\begin{remark}
Analyzing the proof of Theorem 1 we can deduce that taking the assumption $%
\left( a_{n,k}\right) _{k=0}^{n}\in RBVS$ or $\left( a_{n,k}\right)
_{k=0}^{n}\in HBVS$ $\ $instead of $\left( a_{n,k}\right) _{k=0}^{n}\in
MRBVS $ or $\left( a_{n,k}\right) _{k=0}^{n}\in MHBVS\ $, respectively,\ we
obtain the results from the paper \cite{LS}.
\end{remark}

\begin{remark}
In the case $p\geq 1\left( \text{specially if \ }p=1\right) $ we can suppose
that the expression\ $t^{-\beta }\omega \left( t\right) $ is nondecreasing
in $t$ instead of the assumption $\beta <1-\frac{1}{p}.$
\end{remark}

\section{Auxiliary results}

We begin\ this section by some notations following A. Zygmund (\cite{Z}). It
is clear that 
\begin{equation*}
S_{k}f\left( x\right) =\frac{1}{\pi }\int_{-\pi }^{\pi }f\left( x+t\right)
D_{k}\left( t\right) dt
\end{equation*}
and 
\begin{equation*}
T_{n,A}^{\text{ }}f\left( x\right) =\frac{1}{\pi }\int_{-\pi }^{\pi }f\left(
x+t\right) \sum_{k=0}^{n}a_{n,k}D_{k}\left( t\right) dt,
\end{equation*}
where 
\begin{equation*}
D_{k}\left( t\right) =\frac{1}{2}+\sum_{\nu =1}^{k}\cos \nu t=\left\{ 
\begin{array}{c}
\frac{\sin \frac{\left( 2k+1\right) t}{2}}{2\sin \frac{t}{2}}\text{ \ \ for
\ }t\neq 2l\pi , \\ 
k+\frac{1}{2}\text{ \ \ for \ \ }t=2l\pi .%
\end{array}
\right.
\end{equation*}
Hence 
\begin{equation*}
T_{n,A}^{\text{ }}f\left( x\right) -f\left( x\right) =\frac{1}{\pi }%
\int_{0}^{\pi }\varphi _{x}\left( t\right) \sum_{k=0}^{n}a_{n,k}D_{k}\left(
t\right) dt.
\end{equation*}
Next, we present the known estimates for the Dirichlet kernel.

\begin{lemma}
\cite{Z} If \ $0<\left\vert t\right\vert \leq \pi $ then 
\begin{equation*}
\left\vert D_{k}\left( t\right) \right\vert \leq \frac{\pi }{\left\vert
t\right\vert }\text{ \ }
\end{equation*}%
and, for any real $t,$ we have 
\begin{equation*}
\left\vert D_{k}\left( t\right) \right\vert \leq k+1\text{ \ }.
\end{equation*}
\end{lemma}

More complicated estimates we give with proofs.

\begin{lemma}
(cf. \cite{L, LS, S}) If $\left( a_{n,k}\right) _{k=0}^{n}\in MHBVS$, then 
\begin{equation*}
\left\vert \sum_{k=0}^{n}a_{n,k}D_{k}\left( t\right) \right\vert =O\left(
t^{-1}\overline{A}_{n,n-2\tau }\right) \text{,}
\end{equation*}%
and if $\left( a_{n,k}\right) _{k=0}^{n}\in MRBVS$ , then 
\begin{equation*}
\left\vert \sum_{k=0}^{n}a_{n,k}D_{k}\left( t\right) \right\vert =O\left(
t^{-1}A_{n,\tau }\right) ,
\end{equation*}%
for $\frac{2\pi }{n}\leq t\leq \pi $\ $\left( n=2,3,...\right) ,$ where $%
\tau =\left[ \pi /t\right] .$
\end{lemma}

\begin{proof}
Using partial summation 
\begin{equation*}
\left\vert \sum_{k=0}^{n}a_{n,k}\sin \frac{\left( 2k+1\right) t}{2}\sin 
\frac{t}{2}\right\vert 
\end{equation*}%
\begin{equation*}
\leq tA_{n,\tau }+\frac{1}{2}\sum_{k=\tau -1}^{n-1}\left\vert
a_{n,k}-a_{n,k+1}\right\vert \left\vert \sum_{l=\tau -1}^{k}2\sin \frac{%
\left( 2l+1\right) t}{2}\sin \frac{t}{2}\right\vert 
\end{equation*}%
\begin{equation*}
+\frac{1}{2}a_{n,n}\left\vert \sum_{l=\tau -1}^{n}2\sin \frac{\left(
2l+1\right) t}{2}\sin \frac{t}{2}\right\vert 
\end{equation*}%
\begin{equation*}
\leq tA_{n,\tau }+\frac{1}{2}\sum_{k=\tau -1}^{n-1}\left\vert
a_{n,k}-a_{n,k+1}\right\vert \left\vert \sum_{l=\tau -1}^{k}\left( \cos
lt-\cos \left( l+1\right) t\right) \right\vert 
\end{equation*}%
\begin{equation*}
+\frac{1}{2}a_{n,n}\left\vert \sum_{l=\tau -1}^{n}\left( \cos lt-\cos \left(
l+1\right) t\right) \right\vert 
\end{equation*}%
\begin{equation*}
\leq tA_{n,\tau }+\frac{1}{2}\sum_{k=\tau -1}^{n-1}\left\vert
a_{n,k}-a_{n,k+1}\right\vert \left\vert \cos \left( \tau -1\right) t-\cos
\left( k+1\right) t\right\vert 
\end{equation*}%
\begin{equation*}
+\frac{1}{2}a_{n,n}\left\vert \cos \left( \tau -1\right) t-\cos \left(
n+1\right) t\right\vert 
\end{equation*}%
\begin{equation*}
\leq tA_{n,\tau }+\sum_{k=\tau }^{n-1}\left\vert
a_{n,k}-a_{n,k+1}\right\vert +a_{n,n}
\end{equation*}%
or 
\begin{eqnarray*}
&&\left\vert \sum_{k=0}^{n}a_{n,k}\sin \frac{\left( 2k+1\right) t}{2}\sin 
\frac{t}{2}\right\vert  \\
&\leq &t\overline{A}_{n,n-\tau }+\frac{1}{2}\sum_{k=0}^{n-\tau -1}\left\vert
a_{n,k}-a_{n,k+1}\right\vert \left\vert \sum_{l=0}^{k}2\sin \frac{\left(
2l+1\right) t}{2}\sin \frac{t}{2}\right\vert  \\
&&+\frac{1}{2}a_{n,n-\tau }\text{ }\left\vert \sum_{l=0}^{n-\tau }2\sin 
\frac{\left( 2l+1\right) t}{2}\sin \frac{t}{2}\right\vert +a_{n,n-\tau } \\
&\leq &t\overline{A}_{n,n-\tau }+\sum_{k=0}^{n-\tau -1}\left\vert
a_{n,k}-a_{n,k+1}\right\vert +2a_{n,n-\tau }\text{.}
\end{eqnarray*}%
Because $\left( a_{n,k}\right) _{k=0}^{n}\in MRBVS$\ we have 
\begin{eqnarray*}
a_{n,s+1}-a_{n,m} &\leq &\left\vert a_{n,m}-a_{n,s+1}\right\vert \leq
\sum_{k=m}^{s}\left\vert a_{n,k}-a_{n,k+1}\right\vert  \\
&\leq &\sum_{k=r}^{n-1}\left\vert a_{n,k}-a_{n,k+1}\right\vert \ll \frac{1}{%
r+1}\sum_{k\geq r/2}^{r}a_{n,k}\text{ \ \ }\left( 0\leq r\leq m\leq
s<n\right) ,
\end{eqnarray*}%
whence 
\begin{equation*}
a_{n,s+1}\ll a_{n,m}+\frac{1}{r+1}\sum_{k\geq r/2}^{r}a_{n,k}\text{ \ }%
\left( 0\leq r\leq m\leq s<n\right) 
\end{equation*}%
and therefore 
\begin{eqnarray*}
&&\left\vert \sum_{k=0}^{n}a_{n,k}\sin \frac{\left( 2k+1\right) t}{2}\sin 
\frac{t}{2}\right\vert  \\
&\ll &tA_{n,\tau }+\frac{1}{\tau +1}\sum_{k\geq \tau /2}^{\tau }a_{n,k}+%
\frac{1}{\tau +1}\sum_{k\geq \tau /2}^{\tau }a_{n,n} \\
&\ll &tA_{n,\tau }+\frac{1}{\tau +1}\sum_{k\geq \tau /2}^{\tau }a_{n,k}+%
\frac{1}{\tau +1}\sum_{m\geq \tau /2}^{\tau }\left( a_{n,m}+\frac{1}{\tau +1}%
\sum_{k\geq \tau /4}^{\tau /2}a_{n,k}\right) \ll tA_{n,\tau }.
\end{eqnarray*}%
Analogously, the relation $\left( a_{n,k}\right) _{k=0}^{n}\in MHBVS$
implies 
\begin{eqnarray*}
a_{n,s}-a_{n,m} &\leq &\left\vert a_{n,m}-a_{n,s}\right\vert \leq
\sum_{k=m}^{s-1}\left\vert a_{n,k}-a_{n,k+1}\right\vert  \\
&\leq &\sum_{k=0}^{r-1}\left\vert a_{n,k}-a_{n,k+1}\right\vert  \\
&\ll &\frac{1}{n-r+1}\sum_{k=r}^{n}a_{n,k}\text{ \ }\left( 0\leq m<s\leq
r\leq n\right) 
\end{eqnarray*}%
and 
\begin{equation*}
a_{n,s}\ll a_{n,m}+\frac{1}{n-r+1}\sum_{k=r}^{n}a_{n,k}\text{ \ }\left(
0\leq m<s\leq r\leq n\right) ,
\end{equation*}%
whence 
\begin{eqnarray*}
&&\left\vert \sum_{k=0}^{n}a_{n,k}\sin \frac{\left( 2k+1\right) t}{2}\sin 
\frac{t}{2}\right\vert  \\
&\ll &t\overline{A}_{n,n-\tau }+\frac{1}{\tau +1}\sum_{k=n-\tau
}^{n}a_{n,k}+a_{n,n-\tau }\frac{1}{\tau }\sum_{m=n-2\tau }^{n-\tau -1}1 \\
&\ll &t\overline{A}_{n,n-\tau }+\frac{1}{\tau }\sum_{m=n-2\tau }^{n-\tau
-1}\left( a_{n,m}+\frac{1}{\tau +1}\sum_{k=n-\tau }^{n}a_{n,k}\right)  \\
&\leq &2t\overline{A}_{n,n-\tau }+\frac{1}{\tau }\sum_{k=n-2\tau }^{n-\tau
-1}a_{n,k}\ll t\overline{A}_{n,n-2\tau }\text{ \ \ }\left( \tau \leq \frac{n%
}{2}\right) .
\end{eqnarray*}

Thus our proof is complete.
\end{proof}

\begin{lemma}
Let \ $f\in L^{p}$ $\left( 1<p<\infty \right) $, $\beta \geq 0$ and $\ 0\leq
\gamma <\beta +\frac{1}{p}$ . If 
\begin{equation}
\left\{ \int_{\frac{\pi }{n+1}}^{\pi }\left( \frac{\left\vert \varphi
_{x}\left( t\right) \right\vert }{t^{\gamma }\omega \left( t\right) }\right)
^{p}\sin ^{\beta p}\frac{t}{2}dt\right\} ^{\frac{1}{p}}=O_{x}\left( \left(
n+1\right) ^{\gamma }\right) ,  \label{Q}
\end{equation}%
then 
\begin{equation*}
\sum\limits_{m=1}^{n}(m+1)^{\beta +1-\frac{2}{q}}\left\{ \int_{\frac{\pi }{%
m+1}}^{\frac{\pi }{m}}\left( \frac{\left\vert \varphi _{x}\left( t\right)
\right\vert }{\omega \left( t\right) }\right) ^{p}\sin ^{\beta p}\frac{t}{2}%
dt\right\} ^{1/p}=O_{x}\left( \left( n+1\right) ^{\beta +1/p}\right) .
\end{equation*}
\end{lemma}

\begin{proof}
Let (\ref{Q}) holds. Then, using the H\"{o}lder inequality $\left( \frac{1}{p%
}+\frac{1}{q}=1\right) $, we have 
\begin{eqnarray*}
&&\sum\limits_{m=1}^{n}(m+1)^{\beta +1-\frac{2}{q}}\left\{ \int_{\frac{\pi }{%
m+1}}^{\frac{\pi }{m}}\left( \frac{\left\vert \varphi _{x}\left( t\right)
\right\vert }{\omega \left( t\right) }\right) ^{p}\sin ^{\beta p}\frac{t}{2}%
dt\right\} ^{1/p} \\
&\ll &\sum\limits_{m=1}^{n}(m+1)^{\beta -\gamma +1-\frac{2}{q}}\left\{ \int_{%
\frac{\pi }{m+1}}^{\frac{\pi }{m}}\left( \frac{\left\vert \varphi _{x}\left(
t\right) \right\vert }{t^{\gamma }\omega \left( t\right) }\right) ^{p}\sin
^{\beta p}\frac{t}{2}dt\right\} ^{1/p} \\
&\leq &\left\{ \sum\limits_{m=1}^{n}\left( (m+1)^{\beta -\gamma +1-\frac{2}{q%
}}\right) ^{q}\right\} ^{\frac{1}{q}}\left\{ \sum\limits_{m=1}^{n}\int_{%
\frac{\pi }{m+1}}^{\frac{\pi }{m}}\left( \frac{\left\vert \varphi _{x}\left(
t\right) \right\vert }{t^{\gamma }\omega \left( t\right) }\right) ^{p}\sin
^{\beta p}\frac{t}{2}dt\right\} ^{1/p} \\
&\leq &\left\{ \sum\limits_{m=1}^{n}(m+1)^{\beta q-\gamma q+q-2}\right\} ^{%
\frac{1}{q}}\left\{ \int_{\frac{\pi }{n+1}}^{\pi }\left( \frac{\left\vert
\varphi _{x}\left( t\right) \right\vert }{t^{\gamma }\omega \left( t\right) }%
\right) ^{p}\sin ^{\beta p}\frac{t}{2}dt\right\} ^{\frac{1}{p}} \\
&\ll &\left( n+1\right) ^{\beta +1/p}\text{ \ \ }
\end{eqnarray*}%
when $0\leq \gamma <\beta +\frac{1}{p}.$
\end{proof}

\section{Proofs of the results}

\subsection{Proof of Theorem 1}

Let 
\begin{eqnarray*}
T_{n,A}^{\text{ }}f\left( x\right) -f\left( x\right)  &=&\frac{1}{\pi }%
\int_{0}^{\frac{2\pi }{n}}\varphi _{x}\left( t\right)
\sum_{k=0}^{n}a_{n,k}D_{k}\left( t\right) dt \\
&&+\frac{1}{\pi }\int_{\frac{2\pi }{n}}^{\pi }\varphi _{x}\left( t\right)
\sum_{k=0}^{n}a_{n,k}D_{k}\left( t\right) dt \\
&=&I_{1}+I_{2},
\end{eqnarray*}%
where $n\geq 2$. Then 
\begin{equation*}
\left\vert T_{n,A}^{\text{ }}f\left( x\right) -f\left( x\right) \right\vert
\leq \left\vert I_{1}\right\vert +\left\vert I_{2}\right\vert .
\end{equation*}%
By the H\"{o}lder inequality $\left( \frac{1}{p}+\frac{1}{q}=1\right) ,$
Lemma 1 and (\ref{2.7})$,$ for $\beta <1-\frac{1}{p},$%
\begin{eqnarray*}
\left\vert I_{1}\right\vert  &\leq &\frac{\left( n+1\right) }{\pi }\int_{0}^{%
\frac{2\pi }{n}}\left\vert \varphi _{x}\left( t\right) \right\vert dt \\
&\leq &\frac{\left( n+1\right) }{\pi }\left\{ \int_{0}^{\frac{2\pi }{n}}%
\left[ \frac{\left\vert \varphi _{x}\left( t\right) \right\vert }{\omega
\left( t\right) }\sin ^{\beta }\frac{t}{2}\right] ^{p}dt\right\} ^{\frac{1}{p%
}}\left\{ \int_{0}^{\frac{2\pi }{n}}\left[ \frac{\omega \left( t\right) }{%
\sin ^{\beta }\frac{t}{2}}\right] ^{q}dt\right\} ^{\frac{1}{q}}
\end{eqnarray*}%
\begin{equation*}
\ll \left( n+1\right) ^{1-\frac{1}{p}}\omega \left( \frac{2\pi }{n}\right)
\left\{ \int_{0}^{\frac{2\pi }{n}}\left[ \frac{1}{t^{\beta }}\right]
^{q}dt\right\} ^{\frac{1}{q}}\ll \left( n+1\right) ^{\beta }\omega \left( 
\frac{\pi }{n+1}\right) .
\end{equation*}%
Suppose that $\left( a_{n,k}\right) _{k=0}^{n}\in MRBVS$. By the H\"{o}lder
inequality $\left( \frac{1}{p}+\frac{1}{q}=1\right) ,$ Lemma 2 and (\ref{2.6}%
) for an even $n$%
\begin{eqnarray*}
\left\vert I_{2}\right\vert  &\ll &\int_{\frac{2\pi }{n}}^{\pi }\frac{%
\left\vert \varphi _{x}\left( t\right) \right\vert }{t}\sum_{k=0}^{\tau
}a_{n,k}dt=\sum_{m=1}^{\frac{n}{2}-1}\int_{\frac{\pi }{m+1}}^{\frac{\pi }{m}}%
\frac{\left\vert \varphi _{x}\left( t\right) \right\vert }{t}%
\sum_{k=0}^{\tau }a_{n,k}dt \\
&\leq &\sum_{m=1}^{\frac{n}{2}-1}\sum_{k=0}^{m+1}a_{n,k}\int_{\frac{\pi }{m+1%
}}^{\frac{\pi }{m}}\frac{\left\vert \varphi _{x}\left( t\right) \right\vert 
}{t}dt\leq \sum_{m=2}^{\frac{n}{2}}\sum_{k=0}^{m}a_{n,k}\int_{\frac{\pi }{m}%
}^{\frac{\pi }{m-1}}\frac{\left\vert \varphi _{x}\left( t\right) \right\vert 
}{t}dt \\
&=&\sum_{k=2}^{\frac{n}{2}}a_{n,k}\sum_{m=k}^{\frac{n}{2}}\int_{\frac{\pi }{m%
}}^{\frac{\pi }{m-1}}\frac{\left\vert \varphi _{x}\left( t\right)
\right\vert }{t}dt+\sum_{k=0}^{1}a_{n,k}\sum_{m=2}^{\frac{n}{2}}\int_{\frac{%
\pi }{m}}^{\frac{\pi }{m-1}}\frac{\left\vert \varphi _{x}\left( t\right)
\right\vert }{t}dt
\end{eqnarray*}%
\begin{eqnarray*}
&\leq &\sum_{k=2}^{\frac{n}{2}}a_{n,k}\sum_{m=k}^{\frac{n}{2}}\left\{ \int_{%
\frac{\pi }{m}}^{\frac{\pi }{m-1}}\left[ \frac{\left\vert \varphi _{x}\left(
t\right) \right\vert }{\omega \left( t\right) }\sin ^{\beta }\frac{t}{2}%
\right] ^{p}dt\right\} ^{\frac{1}{p}}\left\{ \int_{\frac{\pi }{m}}^{\frac{%
\pi }{m-1}}\left[ \frac{\omega \left( t\right) }{t\sin ^{\beta }\frac{t}{2}}%
\right] ^{q}dt\right\} ^{\frac{1}{q}} \\
&&+\sum_{k=0}^{1}a_{n,k}\sum_{m=2}^{\frac{n}{2}}\left\{ \int_{\frac{\pi }{m}%
}^{\frac{\pi }{m-1}}\left[ \frac{\left\vert \varphi _{x}\left( t\right)
\right\vert }{\omega \left( t\right) }\sin ^{\beta }\frac{t}{2}\right]
^{p}dt\right\} ^{\frac{1}{p}}\left\{ \int_{\frac{\pi }{m}}^{\frac{\pi }{m-1}}%
\left[ \frac{\omega \left( t\right) }{t\sin ^{\beta }\frac{t}{2}}\right]
^{q}dt\right\} ^{\frac{1}{q}} \\
&\leq &\sum_{k=2}^{\frac{n}{2}}a_{n,k}\sum_{m=k}^{\frac{n}{2}}\omega \left( 
\frac{\pi }{m-1}\right) \left\{ \int_{\frac{\pi }{m}}^{\frac{\pi }{m-1}}%
\left[ \frac{\left\vert \varphi _{x}\left( t\right) \right\vert }{\omega
\left( t\right) }\sin ^{\beta }\frac{t}{2}\right] ^{p}dt\right\} ^{\frac{1}{p%
}}\left\{ \int_{\frac{\pi }{m}}^{\frac{\pi }{m-1}}t^{\left( -1-\beta \right)
q}dt\right\} ^{\frac{1}{q}} \\
&&+\sum_{k=0}^{1}a_{n,k}\sum_{m=2}^{\frac{n}{2}}\omega \left( \frac{\pi }{m-1%
}\right) \left\{ \int_{\frac{\pi }{m}}^{\frac{\pi }{m-1}}\left[ \frac{%
\left\vert \varphi _{x}\left( t\right) \right\vert }{\omega \left( t\right) }%
\sin ^{\beta }\frac{t}{2}\right] ^{p}dt\right\} ^{\frac{1}{p}}\left\{ \int_{%
\frac{\pi }{m}}^{\frac{\pi }{m-1}}t^{\left( -1-\beta \right) q}dt\right\} ^{%
\frac{1}{q}} \\
&\ll &\sum_{k=0}^{\frac{n}{2}}a_{n,k}\omega \left( \frac{\pi }{k+1}\right)
\sum_{m=2}^{\frac{n}{2}}m^{\beta +1-\frac{2}{q}}\left\{ \int_{\frac{\pi }{m}%
}^{\frac{\pi }{m-1}}\left[ \frac{\left\vert \varphi _{x}\left( t\right)
\right\vert }{\omega \left( t\right) }\sin ^{\beta }\frac{t}{2}\right]
^{p}dt\right\} ^{\frac{1}{p}}
\end{eqnarray*}%
\begin{equation*}
\ll (n+1)^{\beta +\frac{1}{p}}\sum_{k=0}^{n}a_{n,k}\omega \left( \frac{\pi }{%
k+1}\right) ,
\end{equation*}%
and for an odd $n$ $(n>2)$%
\begin{eqnarray*}
\left\vert I_{2}\right\vert  &\ll &\int_{\frac{2\pi }{n}}^{\pi }\frac{%
\left\vert \varphi _{x}\left( t\right) \right\vert }{t}\sum_{k=0}^{\tau
}a_{n,k}dt \\
&=&\sum_{m=1}^{\left[ \frac{n}{2}\right] -1}\int_{\frac{\pi }{m+1}}^{\frac{%
\pi }{m}}\frac{\left\vert \varphi _{x}\left( t\right) \right\vert }{t}%
\sum_{k=0}^{\tau }a_{n,k}dt+\int_{\frac{2\pi }{n}}^{\frac{\pi }{\left[ \frac{%
n}{2}\right] }}\frac{\left\vert \varphi _{x}\left( t\right) \right\vert }{t}%
\sum_{k=0}^{\tau }a_{n,k}dt=\sum\nolimits_{1}+\sum\nolimits_{2},
\end{eqnarray*}%
the first sum we can estimate similarly like for an even $n$\ and therefore 
\begin{eqnarray*}
\sum\nolimits_{1} &\ll &\sum_{k=0}^{\left[ \frac{n}{2}\right] }a_{n,k}\omega
\left( \frac{\pi }{k+1}\right) \sum_{m=2}^{\left[ \frac{n}{2}\right]
}\left\{ \int_{\frac{\pi }{m}}^{\frac{\pi }{m-1}}\left[ \frac{\left\vert
\varphi _{x}\left( t\right) \right\vert }{\omega \left( t\right) }\sin
^{\beta }\frac{t}{2}\right] ^{p}dt\right\} ^{\frac{1}{p}}\left\{ \int_{\frac{%
\pi }{m}}^{\frac{\pi }{m-1}}t^{\left( -1-\beta \right) q}dt\right\} ^{\frac{1%
}{q}} \\
&\ll &(n+1)^{\beta +\frac{1}{p}}\sum_{k=0}^{n}a_{n,k}\omega \left( \frac{\pi 
}{k+1}\right) .
\end{eqnarray*}%
For the second sum we have 
\begin{equation*}
\sum\nolimits_{2}\leq \sum_{k=0}^{\left[ \frac{n}{2}\right] }a_{n,k}\int_{%
\frac{2\pi }{n}}^{\frac{\pi }{\left[ \frac{n}{2}\right] }}\frac{\left\vert
\varphi _{x}\left( t\right) \right\vert }{t}dt\leq \sum_{k=0}^{\left[ \frac{n%
}{2}\right] }a_{n,k}\int_{\frac{\pi }{\left[ \frac{n}{2}\right] +1}}^{\frac{%
\pi }{\left[ \frac{n}{2}\right] }}\frac{\left\vert \varphi _{x}\left(
t\right) \right\vert }{t}dt
\end{equation*}%
\begin{eqnarray*}
&\leq &\sum_{k=0}^{\left[ \frac{n}{2}\right] }a_{n,k}\left\{ \int_{\frac{\pi 
}{\left[ \frac{n}{2}\right] +1}}^{\frac{\pi }{\left[ \frac{n}{2}\right] }}%
\left[ \frac{\left\vert \varphi _{x}\left( t\right) \right\vert }{\omega
\left( t\right) }\sin ^{\beta }\frac{t}{2}\right] ^{p}dt\right\} ^{\frac{1}{p%
}}\left\{ \int_{\frac{\pi }{\left[ \frac{n}{2}\right] +1}}^{\frac{\pi }{%
\left[ \frac{n}{2}\right] }}\left[ \frac{\omega \left( t\right) }{t\sin
^{\beta }\frac{t}{2}}\right] ^{q}dt\right\} ^{\frac{1}{q}} \\
&\ll &\sum_{k=0}^{\left[ \frac{n}{2}\right] }a_{n,k}\omega \left( \frac{\pi 
}{n+1}\right) \left( \left[ \frac{n}{2}\right] +1\right) ^{1+\beta -\frac{2}{%
q}}\left\{ \int_{\frac{\pi }{\left[ \frac{n}{2}\right] +1}}^{\frac{\pi }{%
\left[ \frac{n}{2}\right] }}\left[ \frac{\left\vert \varphi _{x}\left(
t\right) \right\vert }{\omega \left( t\right) }\sin ^{\beta }\frac{t}{2}%
\right] ^{p}dt\right\} ^{\frac{1}{p}} \\
&\ll &\sum_{k=0}^{\left[ \frac{n}{2}\right] }a_{n,k}\left( \left[ \frac{n}{2}%
\right] +1\right) ^{\beta +\frac{1}{p}}\omega \left( \frac{\pi }{n+1}\right)
\ll (n+1)^{\beta +\frac{1}{p}}\sum_{k=0}^{n}a_{n,k}\omega \left( \frac{\pi }{%
k+1}\right) .
\end{eqnarray*}%
Consequently, let $\left( a_{n,k}\right) _{k=0}^{n}\in MHBVS$. By the H\"{o}%
lder inequality $\left( \frac{1}{p}+\frac{1}{q}=1\right) ,$ Lemma 2 and (\ref%
{2.6}) for an even $n$%
\begin{eqnarray*}
\left\vert I_{2}\right\vert  &\ll &\int_{\frac{2\pi }{n}}^{\pi }\frac{%
\left\vert \varphi _{x}\left( t\right) \right\vert }{t}\sum_{k=n-2\tau
}^{n}a_{n,k}dt=\sum_{m=1}^{\frac{n}{2}-1}\int_{\frac{\pi }{m+1}}^{\frac{\pi 
}{m}}\frac{\left\vert \varphi _{x}\left( t\right) \right\vert }{t}%
\sum_{k=0}^{\tau }a_{n,n-2k}dt \\
&\leq &\sum_{k=2}^{\frac{n}{2}}a_{n,n-2k}\sum_{m=k}^{\frac{n}{2}}\int_{\frac{%
\pi }{m}}^{\frac{\pi }{m-1}}\frac{\left\vert \varphi _{x}\left( t\right)
\right\vert }{t}dt+\sum_{k=0}^{1}a_{n,n-2k}\sum_{m=2}^{\frac{n}{2}}\int_{%
\frac{\pi }{m}}^{\frac{\pi }{m-1}}\frac{\left\vert \varphi _{x}\left(
t\right) \right\vert }{t}dt
\end{eqnarray*}%
\begin{eqnarray*}
&\leq &\sum_{k=2}^{\frac{n}{2}}a_{n,n-2k}\sum_{m=k}^{\frac{n}{2}}\left\{
\int_{\frac{\pi }{m}}^{\frac{\pi }{m-1}}\left[ \frac{\left\vert \varphi
_{x}\left( t\right) \right\vert }{\omega \left( t\right) }\sin ^{\beta }%
\frac{t}{2}\right] ^{p}dt\right\} ^{\frac{1}{p}}\left\{ \int_{\frac{\pi }{m}%
}^{\frac{\pi }{m-1}}\left[ \frac{\omega \left( t\right) }{t\sin ^{\beta }%
\frac{t}{2}}\right] ^{q}dt\right\} ^{\frac{1}{q}} \\
&&+\sum_{k=0}^{1}a_{n,n-2k}\sum_{m=2}^{\frac{n}{2}}\left\{ \int_{\frac{\pi }{%
m}}^{\frac{\pi }{m-1}}\left[ \frac{\left\vert \varphi _{x}\left( t\right)
\right\vert }{\omega \left( t\right) }\sin ^{\beta }\frac{t}{2}\right]
^{p}dt\right\} ^{\frac{1}{p}}\left\{ \int_{\frac{\pi }{m}}^{\frac{\pi }{m-1}}%
\left[ \frac{\omega \left( t\right) }{t\sin ^{\beta }\frac{t}{2}}\right]
^{q}dt\right\} ^{\frac{1}{q}} \\
&\ll &\sum_{k=0}^{\frac{n}{2}}a_{n,n-2k}\omega \left( \frac{\pi }{k+1}%
\right) \sum_{m=2}^{\frac{n}{2}}m^{\beta +1-\frac{2}{q}}\left\{ \int_{\frac{%
\pi }{m}}^{\frac{\pi }{m-1}}\left[ \frac{\left\vert \varphi _{x}\left(
t\right) \right\vert }{\omega \left( t\right) }\sin ^{\beta }\frac{t}{2}%
\right] ^{p}dt\right\} ^{\frac{1}{p}}
\end{eqnarray*}%
\begin{equation*}
\ll n^{\beta +\frac{1}{p}}\sum_{k=0}^{\frac{n}{2}}a_{n,n-2k}\omega \left( 
\frac{\pi }{k+1}\right) \leq (n+1)^{\beta +\frac{1}{p}%
}\sum_{k=0}^{n}a_{n,n-k}\omega \left( \frac{\pi }{k+1}\right) ,
\end{equation*}%
and for an odd $n$ $(n>2)$%
\begin{eqnarray*}
\left\vert I_{2}\right\vert  &\ll &\int_{\frac{2\pi }{n}}^{\pi }\frac{%
\left\vert \varphi _{x}\left( t\right) \right\vert }{t}\sum_{k=n-2\tau
}^{n}a_{n,k}dt \\
&=&\sum_{m=1}^{\left[ \frac{n}{2}\right] -1}\int_{\frac{\pi }{m+1}}^{\frac{%
\pi }{m}}\frac{\left\vert \varphi _{x}\left( t\right) \right\vert }{t}%
\sum_{k=0}^{\tau }a_{n,n-2k}dt+\int_{\frac{2\pi }{n}}^{\frac{\pi }{\left[ 
\frac{n}{2}\right] }}\frac{\left\vert \varphi _{x}\left( t\right)
\right\vert }{t}\sum_{k=0}^{\tau
}a_{n,n-2k}dt=\sum\nolimits_{1}+\sum\nolimits_{2}.
\end{eqnarray*}%
The first sum we can estimate similarly like above 
\begin{eqnarray*}
\sum\nolimits_{1} &\ll &\sum_{k=0}^{\left[ \frac{n}{2}\right]
}a_{n,n-2k}\omega \left( \frac{\pi }{k+1}\right) \sum_{m=2}^{\left[ \frac{n}{%
2}\right] }\left\{ \int_{\frac{\pi }{m}}^{\frac{\pi }{m-1}}\left[ \frac{%
\left\vert \varphi _{x}\left( t\right) \right\vert }{\omega \left( t\right) }%
\sin ^{\beta }\frac{t}{2}\right] ^{p}dt\right\} ^{\frac{1}{p}}\left\{ \int_{%
\frac{\pi }{m}}^{\frac{\pi }{m-1}}t^{\left( -1-\beta \right) q}dt\right\} ^{%
\frac{1}{q}} \\
&\ll &(n+1)^{\beta +\frac{1}{p}}\sum_{k=0}^{n}a_{n,n-k}\omega \left( \frac{%
\pi }{k+1}\right) .
\end{eqnarray*}%
For the second sum we have 
\begin{equation*}
\sum\nolimits_{2}\leq \sum_{k=0}^{\left[ \frac{n}{2}\right] }a_{n,n-2k}\int_{%
\frac{\pi }{\left[ \frac{n}{2}\right] +1}}^{\frac{\pi }{\left[ \frac{n}{2}%
\right] }}\frac{\left\vert \varphi _{x}\left( t\right) \right\vert }{t}dt
\end{equation*}%
\begin{eqnarray*}
&\leq &\sum_{k=0}^{\left[ \frac{n}{2}\right] }a_{n,n-2k}\left\{ \int_{\frac{%
\pi }{\left[ \frac{n}{2}\right] +1}}^{\frac{\pi }{\left[ \frac{n}{2}\right] }%
}\left[ \frac{\left\vert \varphi _{x}\left( t\right) \right\vert }{\omega
\left( t\right) }\sin ^{\beta }\frac{t}{2}\right] ^{p}dt\right\} ^{\frac{1}{p%
}}\left\{ \int_{\frac{\pi }{\left[ \frac{n}{2}\right] +1}}^{\frac{\pi }{%
\left[ \frac{n}{2}\right] }}\left[ \frac{\omega \left( t\right) }{t\sin
^{\beta }\frac{t}{2}}\right] ^{q}dt\right\} ^{\frac{1}{q}} \\
&\ll &\sum_{k=0}^{\left[ \frac{n}{2}\right] }a_{n,n-2k}\left( \left[ \frac{n%
}{2}\right] +1\right) ^{\beta +\frac{1}{p}}\omega \left( \frac{\pi }{n+1}%
\right) \ll (n+1)^{\beta +\frac{1}{p}}\sum_{k=0}^{n}a_{n,n-k}\omega \left( 
\frac{\pi }{k+1}\right) .
\end{eqnarray*}%
Collecting these estimates we obtain the desired result. $\blacksquare $

\subsection{Proof of Theorem 2}

Similarly to the proof of Theorem 1, we have 
\begin{equation*}
\left\vert T_{n,A}^{\text{ }}f\left( x\right) -f\left( x\right) \right\vert
\leq \left\vert I_{1}\right\vert +\left\vert I_{2}\right\vert .
\end{equation*}%
and for $\beta <1-\frac{1}{p},$%
\begin{equation*}
\left\vert I_{1}\right\vert \ll \left( n+1\right) ^{\beta }\omega \left( 
\frac{\pi }{n+1}\right) .
\end{equation*}%
Suppose that $\left( a_{n,k}\right) _{k=0}^{n}\in MRBVS$. By the H\"{o}lder
inequality $\left( \frac{1}{p}+\frac{1}{q}=1\right) ,$ Lemma 2 and (\ref{2.8}%
) for an even $n$ 
\begin{equation*}
\left\vert I_{2}\right\vert \ll \int_{\frac{2\pi }{n}}^{\pi }\frac{%
\left\vert \varphi _{x}\left( t\right) \right\vert }{t}\sum_{k=0}^{\tau
}a_{n,k}dt=\sum_{m=1}^{\frac{n}{2}-1}\int_{\frac{\pi }{m+1}}^{\frac{\pi }{m}}%
\frac{\left\vert \varphi _{x}\left( t\right) \right\vert }{t}%
\sum_{k=0}^{\tau }a_{n,k}dt
\end{equation*}%
\begin{equation*}
\ll \sum_{k=0}^{\frac{n}{2}}a_{n,k}\omega \left( \frac{\pi }{k+1}\right)
\sum_{m=2}^{\frac{n}{2}}m^{\beta +1-\frac{2}{q}}\left\{ \int_{\frac{\pi }{m}%
}^{\frac{\pi }{m-1}}\left[ \frac{\left\vert \varphi _{x}\left( t\right)
\right\vert }{\omega \left( t\right) }\sin ^{\beta }\frac{t}{2}\right]
^{p}dt\right\} ^{\frac{1}{p}}
\end{equation*}%
\begin{equation*}
\ll (n+1)^{\beta }\sum_{k=0}^{n}a_{n,k}\omega \left( \frac{\pi }{k+1}\right)
,
\end{equation*}%
and for an odd $n$ $(n>2)$%
\begin{equation*}
\left\vert I_{2}\right\vert \ll \sum_{m=1}^{\left[ \frac{n}{2}\right]
-1}\int_{\frac{\pi }{m+1}}^{\frac{\pi }{m}}\frac{\left\vert \varphi
_{x}\left( t\right) \right\vert }{t}\sum_{k=0}^{\tau }a_{n,k}dt+\int_{\frac{%
2\pi }{n}}^{\frac{\pi }{\left[ \frac{n}{2}\right] }}\frac{\left\vert \varphi
_{x}\left( t\right) \right\vert }{t}\sum_{k=0}^{\tau
}a_{n,k}dt=\sum\nolimits_{1}+\sum\nolimits_{2}.
\end{equation*}%
For the first sum we have 
\begin{eqnarray*}
\sum\nolimits_{1} &\ll &\sum_{k=0}^{\left[ \frac{n}{2}\right] }a_{n,k}\omega
\left( \frac{\pi }{k+1}\right) \sum_{m=2}^{\left[ \frac{n}{2}\right]
}\left\{ \int_{\frac{\pi }{m}}^{\frac{\pi }{m-1}}\left[ \frac{\left\vert
\varphi _{x}\left( t\right) \right\vert }{\omega \left( t\right) }\sin
^{\beta }\frac{t}{2}\right] ^{p}dt\right\} ^{\frac{1}{p}}\left\{ \int_{\frac{%
\pi }{m}}^{\frac{\pi }{m-1}}t^{\left( -1-\beta \right) q}dt\right\} ^{\frac{1%
}{q}} \\
&\ll &(n+1)^{\beta }\sum_{k=0}^{n}a_{n,k}\omega \left( \frac{\pi }{k+1}%
\right) 
\end{eqnarray*}%
and for the second one 
\begin{equation*}
\sum\nolimits_{2}\leq \sum_{k=0}^{\left[ \frac{n}{2}\right] }a_{n,k}\int_{%
\frac{\pi }{\left[ \frac{n}{2}\right] +1}}^{\frac{\pi }{\left[ \frac{n}{2}%
\right] }}\frac{\left\vert \varphi _{x}\left( t\right) \right\vert }{t}dt
\end{equation*}%
\begin{eqnarray*}
&\leq &\sum_{k=0}^{\left[ \frac{n}{2}\right] }a_{n,k}\left\{ \int_{\frac{\pi 
}{\left[ \frac{n}{2}\right] +1}}^{\frac{\pi }{\left[ \frac{n}{2}\right] }}%
\left[ \frac{\left\vert \varphi _{x}\left( t\right) \right\vert }{\omega
\left( t\right) }\sin ^{\beta }\frac{t}{2}\right] ^{p}dt\right\} ^{\frac{1}{p%
}}\left\{ \int_{\frac{\pi }{\left[ \frac{n}{2}\right] +1}}^{\frac{\pi }{%
\left[ \frac{n}{2}\right] }}\left[ \frac{\omega \left( t\right) }{t\sin
^{\beta }\frac{t}{2}}\right] ^{q}dt\right\} ^{\frac{1}{q}} \\
&\ll &\sum_{k=0}^{\left[ \frac{n}{2}\right] }a_{n,k}\omega \left( \frac{\pi 
}{n+1}\right) (\left[ \frac{n}{2}\right] +1)^{1+\beta -\frac{2}{q}}\left\{
\int_{\frac{\pi }{\left[ \frac{n}{2}\right] +1}}^{\frac{\pi }{\left[ \frac{n%
}{2}\right] }}\left[ \frac{\left\vert \varphi _{x}\left( t\right)
\right\vert }{\omega \left( t\right) }\sin ^{\beta }\frac{t}{2}\right]
^{p}dt\right\} ^{\frac{1}{p}} \\
&\ll &(n+1)^{\beta }\sum_{k=0}^{n}a_{n,k}\omega \left( \frac{\pi }{k+1}%
\right) .
\end{eqnarray*}%
Consequently, let $\left( a_{n,k}\right) _{k=0}^{n}\in MHBVS$. By the H\"{o}%
lder inequality $\left( \frac{1}{p}+\frac{1}{q}=1\right) ,$ Lemma 2 and (\ref%
{2.8}) for an even $n$%
\begin{equation*}
\left\vert I_{2}\right\vert \ll \int_{\frac{2\pi }{n}}^{\pi }\frac{%
\left\vert \varphi _{x}\left( t\right) \right\vert }{t}\sum_{k=n-2\tau
}^{n}a_{n,k}dt=\sum_{m=1}^{\frac{n}{2}-1}\int_{\frac{\pi }{m+1}}^{\frac{\pi 
}{m}}\frac{\left\vert \varphi _{x}\left( t\right) \right\vert }{t}%
\sum_{k=0}^{\tau }a_{n,n-2k}dt
\end{equation*}%
\begin{equation*}
\ll \sum_{k=0}^{\frac{n}{2}}a_{n,n-2k}\omega \left( \frac{\pi }{k+1}\right)
\sum_{m=2}^{\frac{n}{2}}m^{\beta +1-\frac{2}{q}}\left\{ \int_{\frac{\pi }{m}%
}^{\frac{\pi }{m-1}}\left[ \frac{\left\vert \varphi _{x}\left( t\right)
\right\vert }{\omega \left( t\right) }\sin ^{\beta }\frac{t}{2}\right]
^{p}dt\right\} ^{\frac{1}{p}}
\end{equation*}%
\begin{equation*}
\ll n^{\beta }\sum_{k=0}^{\frac{n}{2}}a_{n,n-2k}\omega \left( \frac{\pi }{k+1%
}\right) \ll (n+1)^{\beta }\sum_{k=0}^{n}a_{n,n-k}\omega \left( \frac{\pi }{%
k+1}\right) ,
\end{equation*}%
and for an odd $n$ $(n>2)$%
\begin{equation*}
\left\vert I_{2}\right\vert \ll \sum_{m=1}^{\left[ \frac{n}{2}\right]
-1}\int_{\frac{\pi }{m+1}}^{\frac{\pi }{m}}\frac{\left\vert \varphi
_{x}\left( t\right) \right\vert }{t}\sum_{k=0}^{\tau }a_{n,n-2k}dt+\int_{%
\frac{2\pi }{n}}^{\frac{\pi }{\left[ \frac{n}{2}\right] }}\frac{\left\vert
\varphi _{x}\left( t\right) \right\vert }{t}\sum_{k=0}^{\tau
}a_{n,n-2k}dt=\sum\nolimits_{1}+\sum\nolimits_{2}.
\end{equation*}%
Analogously as above 
\begin{eqnarray*}
\sum\nolimits_{1} &\ll &\sum_{k=0}^{\left[ \frac{n}{2}\right]
}a_{n,n-2k}\omega \left( \frac{\pi }{k+1}\right) \sum_{m=2}^{\left[ \frac{n}{%
2}\right] }\left\{ \int_{\frac{\pi }{m}}^{\frac{\pi }{m-1}}\left[ \frac{%
\left\vert \varphi _{x}\left( t\right) \right\vert }{\omega \left( t\right) }%
\sin ^{\beta }\frac{t}{2}\right] ^{p}dt\right\} ^{\frac{1}{p}}\left\{ \int_{%
\frac{\pi }{m}}^{\frac{\pi }{m-1}}t^{\left( -1-\beta \right) q}dt\right\} ^{%
\frac{1}{q}} \\
&\ll &(n+1)^{\beta }\sum_{k=0}^{n}a_{n,n-k}\omega \left( \frac{\pi }{k+1}%
\right) 
\end{eqnarray*}%
and 
\begin{equation*}
\sum\nolimits_{2}\leq \sum_{k=0}^{\left[ \frac{n}{2}\right] }a_{n,n-2k}\int_{%
\frac{\pi }{\left[ \frac{n}{2}\right] +1}}^{\frac{\pi }{\left[ \frac{n}{2}%
\right] }}\frac{\left\vert \varphi _{x}\left( t\right) \right\vert }{t}dt
\end{equation*}%
\begin{eqnarray*}
&\leq &\sum_{k=0}^{\left[ \frac{n}{2}\right] }a_{n,n-2k}\left\{ \int_{\frac{%
\pi }{\left[ \frac{n}{2}\right] +1}}^{\frac{\pi }{\left[ \frac{n}{2}\right] }%
}\left[ \frac{\left\vert \varphi _{x}\left( t\right) \right\vert }{\omega
\left( t\right) }\sin ^{\beta }\frac{t}{2}\right] ^{p}dt\right\} ^{\frac{1}{p%
}}\left\{ \int_{\frac{\pi }{\left[ \frac{n}{2}\right] +1}}^{\frac{\pi }{%
\left[ \frac{n}{2}\right] }}\left[ \frac{\omega \left( t\right) }{t\sin
^{\beta }\frac{t}{2}}\right] ^{q}dt\right\} ^{\frac{1}{q}} \\
&\ll &\sum_{k=0}^{\left[ \frac{n}{2}\right] }a_{n,n-2k}\left( \left[ \frac{n%
}{2}\right] +1\right) ^{\beta }\omega \left( \frac{\pi }{n+1}\right) \ll
(n+1)^{\beta }\sum_{k=0}^{n}a_{n,n-k}\omega \left( \frac{\pi }{k+1}\right) .
\end{eqnarray*}%
Collecting these estimates we obtain the desired result. $\blacksquare $

\subsection{Proofs of Theorems 3}

The proof is similar to these above. The expressions in the estimates under
the $L^{p}$ norm with respect to $x$ will be like these on the left hand
side of our conditions (\ref{2.7}) and (\ref{2.8}) or (\ref{2.6}) when $\
\beta >0$ or $\beta =0,$ respectively. Since $f\in L^{p}\left( \omega
\right) _{\beta }$, the such norm quantities will always have the same
orders like these on the right hand side of the mentioned conditions.
Therefore the proof follows without any additional assumptions. $%
\blacksquare $


\bigskip

\begin{thebibliography}{99}
\bibitem{BL} P. L. Butzer, R. J. Nessel, Fourier analysis and approximation,
Birkh\"{a}user verlag Besel und Stuttgart, 1971.

\bibitem{DT} M. Dyachenko and S. Tikhonov, Convergence of trigonometric
series with general monotone coefficients, C. R. Acad. Sci. Paris,
345(3)(2007), 123-126

\bibitem{SL} S. Lal, Approximation of functions belonging to the generalized
Lipschitz Class by $C1 Np$ summability method of Fourier series, Appl. Math.
Comput., 209 (2009), 346-350.

\bibitem{L} L. Leindler, On the degree of approximation of continuous
functions, Acta Math. Hungar., 104 (1-2) (2004), 105-113.

\bibitem{L2} L. Leindler, Integrability conditions pertaining to Orlicz
space, J. Inequal. Pure and Appl. Math., 8(2) (2007), Art. 38, 6 pp.

\bibitem{LT} E. Liflyand and S. Tikhonov, A concept of general monotonicity
and applications, Math. Nachr., 284(2011), 1083-1098

\bibitem{LS} W. \L enski, B. Szal, Approximation of functions belonging to
the class $L^{p}\left( \omega \right) _{\beta }$ by linear operators, Acta
Comment. Univ. Tartu. Math., Vol. 13 (2009), 11-24.

\bibitem{S} B. Szal, A note on the uniform convergence and boundedness a
generalized class of sine series, Commentat. Math., 48 1 (2008), 85-94.

\bibitem{T} S.Tikhonov, Trigonometric series with general monotone
coefficients, J. Math. Anal. Appl. 326(1)(2007), 721-735

\bibitem{Z} A. Zygmund, Trigonometric series, Cambridge, 2002.
\end{thebibliography}
\end{document}